\documentclass{amsart}
\usepackage{amsfonts}

\newtheorem{thm}{Theorem}[section]

\newtheorem{lema}[thm]{Lemma}
\newtheorem{prop}[thm]{Proposition}
\theoremstyle{definition}

\theoremstyle{remark}

\numberwithin{equation}{section}
\newcommand{\R}{\mathbb R}
\newcommand{\N}{\mathbb N}
\newcommand{\Z}{\mathbb Z}
\newcommand{\X}{\mathcal{X}}
\newcommand{\ve}{\varepsilon}
\newcommand{\lam}{\lambda}
\newcommand{\Lam}{\Lambda}
\newcommand{\cd}{\rightharpoonup}
\newcommand{\cde}{\stackrel{*}{\rightharpoonup}}
\def\pint{\operatorname {--\!\!\!\!\!\int\!\!\!\!\!--}}

\begin{document}
\title{Eigenvalues homogenization for the fractional $p-$Laplacian operator}

\author[A M Salort]{Ariel Martin Salort }
\address{Departamento de Matem\'atica
 \hfill\break \indent FCEN - Universidad de Buenos Aires and
 \hfill\break \indent   IMAS - CONICET.
\hfill\break \indent Ciudad Universitaria, Pabell\'on I \hfill\break \indent   (1428)
Av. Cantilo s/n. \hfill\break \indent Buenos Aires, Argentina.}
\email{asalort@dm.uba.ar}
\urladdr{http://mate.dm.uba.ar/~asalort}

%35P15  	Estimation of eigenvalues, upper and lower bounds
%35P30  	Nonlinear eigenvalue problems, nonlinear spectral theory
%35B27  	Homogenization; equations in media with periodic structure
\subjclass[2010]{35B27, 35P15, 35P30, 34A08}

\keywords{Eigenvalue homogenization, nonlinear eigenvalues, order of convergence, fractional laplacian}

\begin{abstract}
In this work we study the homogenization for eigenvalues of the fractional $p-$Laplace in a bounded domain both with Dirichlet and Neumann conditions. We obtain the convergence of eigenvalues and the explicit order of the convergence rates.
\end{abstract}

\maketitle
%**********************************************************************************************
%**********************************************************************************************

\section{Introduction}
The purpose of this paper is to study the asymptotic behavior as $\ve\to 0$ of the eigenvalues of the following non-local problem
\begin{align} \label{ecux1}
  \begin{cases}
    (-\Delta )^s_p u =\lam_{p,\ve} \rho_\ve |u|^{p-2}u &\quad \textrm{in } \Omega\subset \R^n\\
    u=0 &\quad \R^n \setminus \Omega
  \end{cases}
\end{align}
where for $\ve>0$, the parameter $\lam_{p,\ve}$ is the eigenvalue and $1<p<\infty$. The weight functions $\rho_\ve$ are positive and bonded away from zero and infinity, i.e., for some constants $\rho_-$ and $\rho_+$ it holds that
\begin{equation} \label{condrho}
  0<\rho_-\leq \rho_\ve(x)\leq \rho_+<\infty \qquad x\in\Omega.
\end{equation}
Here, for $s\in(0,1)$  we denote by $(-\Delta)^s_p$ the fractional $p-$Laplace operator, which is defined as
\begin{align*}
(-\Delta)^s_p u(x)&=c\; p.v. \int_{\R^n} \frac{|u(x)-u(y)|^{p-2}(u(x)-u(y))}{|x-y|^{n+sp}}\;  dy
\end{align*}
where  $c$ is a normalization constant depending only on $n$, $s$ and $p$.

The domain $\Omega$ is assumed to be a bounded and open set in $\R^n$, $n\geq 1$.

As $\ve\to 0$ in \eqref{ecux1}, the following limit problem is obtained
\begin{align} \label{ecux2}
  \begin{cases}
    (-\Delta )^s_p u  =\lam_p \rho(x) |u|^{p-2}u &\quad \textrm{in } \Omega\\
    u=0 &\quad \textrm{in} \; \R^n \setminus \Omega
  \end{cases}
\end{align}
where $\rho(x)$ is the weak* limit in $L^\infty(\Omega)$ as $\ve\to 0$ of the sequence $\{\rho_\ve\}_\ve$.

For each fixed value of $\ve$  it is known that there exists a sequence of variational eigenvalues $\{\lam_{k,p}^\ve\}_{k\geq 1}$ of \eqref{ecux1} such that $\lam_{k,p}^\ve \to \infty$ as $k\to\infty$. Analogously, for the limit problem \eqref{ecux2}, there exists a sequence of variational eigenvalues $\{\lam_{k,p}\}_{k\geq 1}$ such that $\lam_{k,p} \to \infty$ as $k\to\infty$ (see Section \ref{sec2}).

We are interested in studying the behavior of the sequence $\{\lam_{k,p}^\ve\}_{k\geq 1}$ as $\ve \to 0$.

When $s=1$ and $p=2$, \eqref{ecux1} becomes the eigenvalue problem for the Laplacian operator with Dirichlet boundary conditions. This problem has been extensively studied and a complete description of the asymptotic behavior of its spectrum was obtained in the 70's. Boccardo and Marcellini \cite{BM76}, and Kesavan \cite{Ke79} proved that for each fixed $k$,
\begin{equation*}  
\lim_{\ve\to 0} \lam_{k,2}^\ve = \lam_{k,2}.
\end{equation*}
Later on, in \cite{Ch-dP} and \cite{FBPS13} this result was extended to $p-$Laplacian type operators.

One of the purposes of our paper is to extend this results to non-local eigenvalue problems. Our first result  states the convergence of the $k-$th eigenvalue of problem \eqref{ecux1} to the $k-$th eigenvalue of the limit problem \eqref{ecux2} when a general family of weight functions is considered.

\begin{thm} \label{teo0}
Let $\Omega\subset\R^n$ be a open bounded domain and $s\in (0,1)$. Let $\lam_{k,p}^\ve$ and $\lam_{k,p}$ be the $k-$th (variational) eigenvalues of \eqref{ecux1} and \eqref{ecux2}, respectively. Then
\begin{equation}  \label{ecxx}
\lim_{\ve\to 0} \lam_{k,p}^\ve = \lam_{k,p}
\end{equation}
for each fixed $k\geq 1$.
\end{thm}

A slight modification in the arguments in the previous result allow us to deal with the following non-local Neumann eigenvalue problem considered recently in \cite{DPS15}
\begin{align} \label{ecux2.n}
  \begin{cases}
    (-\Delta )^s_p u +|u|^{p-2}u =\Lambda_{p,\ve} \rho_\ve |u|^{p-2}u &\quad \textrm{in } \Omega\\
    u\in W^{s,p}(\Omega)
  \end{cases}
\end{align}
for which, again, the min-max theory provides a sequence of variational eigenvalues tending to $+\infty$ denoted by $\{\Lambda_{k,p}^\ve\}_{k\geq 1}$. Analogously to the Dirichlet case, as $\ve\to 0$, a limit problem is obtained in terms of $\rho(x)$, the weak* limit of $\rho_\ve$ in $L^\infty(\Omega)$,
\begin{align} \label{ecux1.n}
  \begin{cases}
    (-\Delta )^s_p u +|u|^{p-2}u =\Lambda_p \rho(x) |u|^{p-2}u &\quad \textrm{in } \Omega\\
    u\in W^{s,p}(\Omega)
  \end{cases}
\end{align}
which has a sequence of eigenvalues denoted by $\{\Lambda_{k,p}\}_{k\geq 1}$. Here $W^{s,p}(\Omega)$ is a fractional order Sobolev space, which is defined in Sections 2. The corresponding convergence result is stated as follows.

\begin{thm} \label{teo0.n}
Let $\Omega\subset\R^n$ be a open bounded domain and $s\in (0,1)$. Let $\Lambda_{k,p}^\ve$ and $\Lambda_{k,p}$ be the $k-$th (variational) eigenvalues of \eqref{ecux2.n} and \eqref{ecux1.n}, respectively. Then
\begin{equation}  \label{ecxx.n}
\lim_{\ve\to 0} \Lam_{k,p}^\ve = \Lam_{k,p}
\end{equation}
for each fixed $k\geq 1$.
\end{thm}

Homogenization theory dates back to the late sixties with the works of Spagnolo and de Giorgi and it developed very rapidly  during the last two decades. Homogenization theory tries to get a good approximation of a macroscopic behavior of the heterogeneous material by letting the parameter $\ve\to 0$. A case of relevant  importance is the study of periodic homogenization problems due to the many applications to physics and engineering. The main references for the homogenization theory of (local) periodic structures are the books by Bensoussan-Lions-Papanicolaou \cite{BLP78}, Sanchez--Palencia \cite{SP70}, Ole{\u\i}nik-Shamaev-Yosifian \cite{OSY92}  among others.

An interesting issue in the homogenization theory is to estimate the rates of convergence of the eigenvalues in \eqref{ecxx} and \eqref{ecxx.n}, that is, to find bounds for the errors $|\lam_{k,p}^\ve-\lam_{k,p}|$ and $|\Lam_{k,p}^\ve-\Lam_{k,p}|$. Since it is desirable to obtain the explicit dependence on $\ve$ , we restrict  our study to periodic weights, i.e., we consider a family of weight functions $\rho_\ve$  given in terms of a single-bounded $Q-$periodic function $\rho$ in the form
$$
	\rho_\ve(x):=\rho(x/\ve), \qquad \ve>0,
$$ 
$Q$ being the unit cube of $\R^n$. The function $\rho$ is assumed to satisfy the bounds \eqref{condrho}. Under these assumptions it is well-known that 
$$
	\rho_\ve \cde\bar\rho \quad  \mbox{ in } L^\infty(\Omega) \qquad \mbox{ as } \ve\to0,
$$
$\bar\rho$ being the average of $\rho$ on $Q$. 

\medskip

In the local case, the rates of convergence for the eigenvalues of the $p-$Laplace operator were studied in several papers.  The authors in \cite{OSY92} proved  some estimates for the Dirichlet and the Neumann case when $p=2$ by using tools from functional analysis in Hilbert spaces. Assuming that $\Omega$ is a Lipschitz domain they showed  that there exists a constant $C$ depending on $k$ and $\Omega$ such that
$$
	|\lam_{k,2}^\ve - \lam_{k,2}| \leq C\ve^\frac{1}{2}.
$$
Later on, under the same assumptions on $\Omega$  it was proved in \cite{KLZ12} the following bounds for both Dirichlet and Neumann boundary conditions,
$$|\lam_{k,2}^\ve - \lam_{k,2}| \leq C\ve |\log \ve|^{\frac{1}{2}+\gamma}$$
for any $\gamma>0$, $C$ depending on $k$ and $\gamma$. When the domain is more regular ($C^{1,1}$ is enough) in \cite{KLZ13} explicit dependence of the constant $C$ on $k$ was obtained. It was proved that
$$|\lam_{k,2}^\ve - \lam_{k,2}| \leq C\ve k^\frac{3}{n} \ve |\log \ve|^{\frac{1}{2}+\gamma}$$
for any $\gamma>0$, $C$ depending on $\gamma$. In both cases, when the domain $\Omega$ is smooth,  the logarithmic term can be removed.

Finally, in \cite{FBPS13} the results were extended to the local $p-$Laplace operator via non-linear techniques and the dependence on the constant was improved. The authors in \cite{FBPS13} proved that
\begin{equation} \label{cota.loc}
	|\lam_{k,p}^\ve - \lam_{k,p}| \leq Ck^\frac{p+1}{n}\ve, \qquad |\Lam_{k,p}^\ve - \Lam_{k,p}| \leq Ck^\frac{2p}{n}\ve
\end{equation}
where $C$ is a constant independent on $k$ and $\ve$ which can be explicitly computed.

\medskip

Up to our knowledge, no  investigation was made on the homogenization and convergence rates for the weighted fractional $p-$Laplacian eigenvalue problem. In contrast with the $p-$Laplacian operator, the non-local nature of \eqref{ecux1} makes it more difficult to deal with the convergence rates. The main  obstacle is how to manage the boundedness of  fractional norms in order to obtain relations between the variational characterization of eigenvalues.

In the two next results we obtain the rates of the convergence of the eigenvalues of problems \eqref{ecux1} and \eqref{ecux2.n} when periodicity assumptions are made on the weight family.

\begin{thm} \label{teo1}
Let $\Omega\subset\R^n$ be an open and bounded domain and $\rho\in L^\infty(\R^n)$ be a $Q-$periodic function satisfying \eqref{condrho}, $Q$ being the unit cube of  $\R^n$. Let $\lam_{k,p}^\ve$ and $\lam_{k,p}$ be the $k-$th variational eigenvalues of \eqref{ecux1} and \eqref{ecux2}, respectively. Then
$$
	|\lam_{k,p}^\ve - \lam_{k,p}|\le C\ve^s(\mu_{k,p})^{1+\frac{1}{p}}
$$
for every $k\in \N$ and $s\in (0,1)$, $\mu_{k,p}$ being the $k-$th variational eigenvalue of the Dirichlet fractional $p-$laplacian of order $s$. The constant $C$ depends only on $\Omega$, $s$, $n$, $p$ and the bounds of $\rho$.
In the case $p=2$ the previous inequality becomes
$$
	|\lam_{k,2}^\ve - \lam_{k,2}|\le C\ve^s k^\frac{3s}{n}
$$
for every $k\in \N$ and $s\in(0,1)$.
\end{thm}

\begin{thm} \label{teo1.n}
Let $\Omega\subset\R^n$ be an open and bounded set with $C^1$ boundary and $\rho\in L^\infty(\R^n)$ be a $Q-$periodic function satisfying \eqref{condrho}, $Q$ being the unit cube of  $\R^n$. Let $\Lam_{k,p}^\ve$ and $\Lam_{k,p}$ be the $k-$th variational eigenvalues of  \eqref{ecux2.n} and \eqref{ecux1.n}, respectively. 
Then
$$
	|\Lam_{k,p}^\ve - \Lam_{k,p}|\le C\ve^s(\mu_{k,p})^{2}
$$
for every $k\in \N$ and $s\in (\frac{1}{p},1)$, $\mu_{k,p}$ being the $k-$th variational eigenvalue of the Dirichlet fractional $p-$Laplacian or order $s$.  The constant $C$ depends only on $\Omega$, $s$, $n$ and the bounds of $\rho$. In the case $p=2$ the previous inequality becomes
$$
	|\Lam_{k,2}^\ve - \Lam_{k,2}|\le   C\ve^s k^{\frac{4s}{n}}
$$
for every $k\in \N$ and $s\in(\frac{1}{p},1)$.
\end{thm}

Although the rates obtained in the two previous results are  similar, in the Neumann case the range of values of $s$ is smaller, and more assumptions on the boundary of $\Omega$ have to be made. Such restrictions arise from the use of trace arguments in the proof.

Observe that the rates obtained in Theorems \ref{teo1} and \ref{teo1.n} are the natural generalization of the results for the local case stated in \eqref{cota.loc}.
 
This paper is organized as follows: in Section 2 we introduce some definitions and properties of the eigenvalues of non-local problems meanwhile that in Section 3 we prove the results stated before.

\section{Eigenvalues of the fractional $p-$laplacian} \label{sec2}
In this section we present some well-known results about fractional Sobolev spaces and the eigenvalues of non-local problems. For more detailed information we refer to the reader, for instance to \cite{DD12}.

Let $\Omega$ be an open and bounded subset of $\R^n$, $n\geq 1$.  For any   $s\in(0,1)$  and  $p\geq1$ we denote  $W^{s,p}(\Omega)$ the fractional Sobolev space defined as follows
$$
	W^{s,p}(\Omega):=\big\{ u\in L^p(\Omega) : \frac{u(x)-u(y)}{|x-y|^{\frac{n}{p}+s}} \in L^p(\Omega\times\Omega) \big\}
$$
endowed with the norm
$$
	\|u\|_{W^{s,p}(\Omega)}:=(\|u\|_{L^p(\Omega)}^p+[u]_{W^{s,p}(\Omega)}^p)^{\frac{1}{p}}
$$
where $[u]_{W^{s,p}(\Omega)}$ is the so-called Gagliardo semi-norm of $u$ defined as
$$
	[u]_{W^{s,p}(\Omega)}^p=\int_{\Omega \times \Omega} \frac{|u(x)-u(y)|^p}{|x-y|^{n+sp}} \;dx \;dy.
$$

We denote $\X^{s,p}_0(\Omega):=\{u\in W^{s,p}(\Omega) \, : \, u=0 \mbox{ in } \R^n\setminus \Omega\}$. 

 \medskip
 
A useful tool to be used is the following fractional Poincar\'e inequality on cubes of side $\ve$. Here we denote $(v)_U$ the average of the function $v$ on the set $U$.

\begin{lema} \label{poincarelema}
	Let $Q$ be the unit cube in $\R^n$, $n\geq 1$.  Then, for every $u\in W^{s,p}(Q_\ve)$, $1<p<\infty$ we have
	$$
		\| u - (u)_{Q_\ve}\|_{L^p(Q_\ve)}\le c \ve^s  [ u ]_{W^{s,p}(Q_\ve)},
	$$
	where $Q_\ve = \ve Q$ and $c$ is a constant depending only on $n$.
\end{lema}

\begin{proof}
	Given $u\in W^{s,p}(Q_\ve)$, by using  Jensen's inequality it follows that
	\begin{align*}
		\int_{Q_\ve} |u-(u)_{Q_\ve}|^p dx &= \int_{Q_\ve} \left| \pint_{Q_\ve} (u(x)-u(y)) \, dy \ \right|^p  dx\\
		&\leq 
		 \int_{Q_\ve}  \pint_{Q_\ve} |u(x)-u(y)|^p \, dy \,   dx\\
		 &\leq c\ve^{sp} \int_{Q_\ve}  \int_{Q_\ve} \frac{|u(x)-u(y)|^p}{|x-y|^{n+sp}} \, dy \,   dx,
	\end{align*}
	from where the result follows.
\end{proof}
In particular, it is readily seen that the following Poincar\'e-type inequality holds:
	$$
		\| u \|_{L^p(Q_\ve)}\le c \ve^s  [ u ]_{W^{s,p}(Q_\ve)}
	$$
for all $u\in \X_0^{s,p}(Q_\ve)$, from where it follows that $[\cdot]_{W^{s,p}(\Omega)}$ is an equivalent norm  in the space $\X_0^{s,p}$.

Another result we will use is the trace's inequality for fractional spaces proved in \cite{sch}, which it is necessary for our auxiliary computations.
\begin{prop}  \label{traza}
Let $\Omega$ be a bounded $C^1$ domain and $\frac{1}{p}<s\leq 1$. Then
$$
	\|u\|_{W^{s-\frac{1}{p},p}(\partial \Omega)} \leq C \|u\|_{W^{s,p}(\Omega)},
$$
where $C$ is a constant depending on $s$, $p$ and $\Omega$.
\end{prop}

\medskip

Let $\Omega\subset \R^n$ be an open bounded domain. Given a weight function $\rho$ bounded away from zero and infinity, we consider the following  Dirichlet eigenvalue problem
\begin{align} \label{ecup}
    (-\Delta  )^s_p u =\lam \rho |u|^{p-2}u \quad \textrm{in } \Omega, \qquad    u=0 \quad \R^n \setminus \Omega.
\end{align}
Due to the non-locality nature of the problem it is needed to consider the boundary condition not only on $\partial \Omega$ but in $\R^n\setminus \Omega$.

This problem has a variational structure. We say that $u\in \X^{s,p}_0(\Omega)$ is a weak solution of \eqref{ecup} if
$$
	\int_{\R^n \times \R^n} \frac{|u(x)-u(y)|^{p-2}(u(x)-u(y))(v(x)-v(y))}{|x-y|^{n+sp}}  =\lam \int_\Omega \rho (x)|u|^{p-2}uv
$$
for every $v\in \X^{s,p}_0(\Omega)$.

The following non-local Neumann eigenvalue problem was considered recently in \cite{DPS15}.
\begin{align} \label{nn1}
    (-\Delta )^s_p u + |u|^{p-2}u =\Lam \rho(x) |u|^{p-2}u \quad \textrm{in } \Omega, \qquad
    u\in W^{s,p}(\Omega).
\end{align}
In this case, we say that a function $u\in W^{s,p}(\Omega)$ is a weak solution of \eqref{nn1} if it holds that
$$
	\int_{\Omega \times \Omega} \frac{|u(x)-u(y)|^{p-2}(u(x)-u(y))(v(x)-v(y))}{|x-y|^{n+sp}}  +\int_\Omega |u|^{p-2}v =\Lam \int_\Omega \rho (x)|u|^{p-2}v.
$$
for every $v\in W^{s,p}(\Omega)$.

As in \cite{IS}, a non-decreasing sequence of eigenvalues for \eqref{ecup} and \eqref{nn1} can be defined by means of cohomological index. We will denote by $\{\lam_{k,p}\}_{k\geq 1}$ and $\{\Lam_{k,p}\}_{k\geq 1}$ such sequences, respectively. They can be written by using the following inf-sup characterization:
\begin{align} \label{carac}
    \lam_{k,p}= \inf_{\mathcal{C}\in \mathcal{D}_k } \sup_{u\in \mathcal{C}}  \frac{[u]^p_{W^{s,p}(\R^n)}}{\|\rho^{\frac{1}{p}} u\|^p_{L^p(\Omega)}}, 
	    \qquad 
    \Lam_{k,p} = \min_{\mathcal{C}\in \mathcal{\tilde D}_k } \max_{u\in \mathcal{C}}  \frac{\|u\|^p_{W^{s,p}(\Omega)}}{\|\rho^{\frac{1}{p}} u\|^p_{L^p(\Omega)}},
\end{align}
where $\mathcal{D}_k=\{\mathcal{W}\subset \mathcal{X}^{s,p}_0(\Omega)\; : \; i(\mathcal{W})\geq k\}$ and $\mathcal{\tilde D}_k=\{\mathcal{W}\subset W^{s,p}(\Omega)\; : \; i(\mathcal{W})\geq k\}$.
Here $i$ denotes the cohomological index, see for instance \cite{MMP} for the definition and further properties.  These formulas differ from the classical ones by the use of the index instead of the genus, but they coincide in the lineal case $p=2$.

Observe that, since $\mathcal{X}^{s,p}_0(\Omega)\subset  W^{s,p}(\Omega)$, it is straightforward to see that Neumann eigenvalues can be bounded with the Dirichlet ones, i.e., 
\begin{equation} \label{lapn}
	\Lam_{k,p} \leq \lam_{k,p}.
\end{equation}

When the weight function $\rho$ is bounded away from zero and infinity, that is, there exist constant such that $0<\rho_-\leq \rho(x)\leq\rho_+<\infty$ for every $x\in \Omega$, from \eqref{carac} it is easy to see that
\begin{equation}  \label{lapf.2}
	(\rho_-)^{-1} \mu_{k,p} \leq \lam_{k,p}\leq (\rho_-)^{-1}  \mu_{k,p},
\end{equation}
where $\mu_{k,p}$ is the $k-$th eigenvalue of the Dirichlet fractional laplacian, i.e., it satisfies the following equation
\begin{align}  \label{lapf}
    (-\Delta  )^s_p u =\mu |u|^{p-2}u \quad \textrm{in } \Omega, \qquad    u=0 \quad \R^n \setminus \Omega.
\end{align}

For $p> 1$, $s\in (0,1)$ and $sp>n$, the authors in \cite{IS} proved that for $k$ large the following bounds hold
 \begin{equation} \label{asintt}
	c_1 |\Omega|^{-\frac{sp}{n}} k^{\frac{sp-n}{n}}\leq \mu_{k,p} \leq c_2|\Omega|^{-\frac{sp}{n}} k^\frac{np-n+sp}{n}
\end{equation}
for some positive constants $c_1$ and $c_2$ depending on $s$, $p$ and $n$.

The variational characterization of eigenvalues plays a fundamental role in our analysis and the proof of our results since it allows to reduce the eigenvalues convergence to the study of oscillating integrals.

In the linear case $p=2$ the sequence defined in \eqref{carac} coincides with the sequence of variational eigenvalues which uses \emph{dimension} instead of \emph{index} (see for instance \cite{SV13})
\begin{align} \label{carac2}
    \lam_{k,2}= \min_{\mathcal{C}\in \mathcal{D}_k } \max_{u\in \mathcal{C}}  \frac{[u]^2_{W^{s,2}(\R^n)}}{\|\rho^{\frac{1}{2}} u\|^2_{L^2(\Omega)}}, 
	    \qquad 
    \Lam_{k,2} = \min_{\mathcal{C}\in \mathcal{\tilde D}_k } \max_{u\in \mathcal{C}}  \frac{\|u\|^2_{W^{s,2}(\Omega)}}{\|\rho^{\frac{1}{2}} u\|^2_{L^2(\Omega)}},
\end{align}
where $\mathcal{D}_k=\{\mathcal{W}\subset \mathcal{X}^{s,2}_0(\Omega)\; : \; dim\;\mathcal{W}=k\}$ and $\mathcal{\tilde D}_k=\{\mathcal{W}\subset W^{s,2}(\Omega)\; : \; dim\;\mathcal{W}=k\}$. 

In \cite{IS} the authors suspect that the estimates \eqref{asintt} on the eigenvalues \eqref{carac} that they obtained are not optimal. However when $p=2$, by using formulation \eqref{carac2},  precise estimates for this sequence are known. In 1959, Blumenthal and Getoor \cite{BG59} proved a Weyl's formula for $\mu_{k,2}$ in the context of  $s-$stable symmetric processes, whose generators are the fractional Laplacians, more precisely, they proved the following asymptotic formula
\begin{equation*}
	\mu_{k,2} \sim  (4\pi)^s\left(k |\Omega|^{-1} \Gamma(1+\frac{n}{2})\right)^\frac{2s}{n}, \qquad k\to +\infty.
\end{equation*}

Moreover, in \cite{CS05} it was proved that there exists some constant $c$ independent on $k$ such that $c(\tilde \mu_{k,2})^s\leq \mu_{k,2} \leq (\tilde \mu_{k,2})^s$, where $\tilde \mu_{k,2}$ is the $k-$th eigenvalue of the usual Laplacian with Dirichlet boundary conditions on $\partial\Omega$. Since it is well-known that there exist constants $c_1$ and $c_2$ independent on $k$ such that $c_1k^\frac{2}{n}\leq \tilde \mu_{k,2} \leq c_2 k^\frac{2}{n}$ (see for instance \cite{CoHi}), for the case $p=2$, inequality \eqref{lapf.2} reads as
\begin{equation}  \label{lapf.1}
	C_1 k^\frac{2s}{n} \leq \lam_{k,2} \leq C_2 k^\frac{2s}{n}
\end{equation}
where $C_1$ and $C_2$ are two constant independent on $k$ and $s$.

\section{Proof of the results}

The convergence of the sequence of Dirichlet and Neumann eigenvalues is a consequence of the following simple lemma concerning to oscillating integrals. Since periodicity is not assumed on the weight functions, the result does not provide any information about the order of the convergence.

\begin{lema} \label{lema.sin.orden}
	Let $\Omega\subset \R^n$ be a bounded domain. Let $\{g_\varepsilon\}_{\ve>0}$ be a set functions such that   $0<g_- \leq g_\varepsilon \leq g_+<+\infty$ for $g_\pm$ constants and $g_\varepsilon \cd g$ weakly* in $L^\infty(\Omega)$. Then 
	$$\lim_{\varepsilon\to 0} \int_\Omega (g_\varepsilon- g) |u|^p =0$$
	for every $u\in W^{s,p}(\Omega)$, $0<s<1$.
\end{lema}
\begin{proof}
	The weak* convergence of $g_\varepsilon$ in $L^\infty(\Omega)$ says that  $\int_\Omega g_\varepsilon \varphi \to \int_\Omega g \varphi$ for all $\varphi \in L^1(\Omega)$. In particular, since $u\in W^{s,p}(\Omega)$, we have that $|u|^p\in L^1(\Omega)$ and the result is proved.
\end{proof}

\begin{proof}[Proof of Theorem \ref{teo0}]
	Let $\delta>0$ and $\mathcal{C}_{k,\delta}\subset \mathcal{X}_0^{s,p}(\Omega)$ be a set of index greater than $k$ such that
	$$
		\lam_{k,p} = 
			\inf_{C\in \mathcal{D}_k} \sup_{u \in C} \frac{  [u]_{W^{s,p}(\R^n)}^p}{ \int_{\Omega} \rho |u|^p} 
		= 
			\sup_{u \in \mathcal{C}_{k,\delta}} \frac{ [u]_{W^{s,p}(\R^n)}^p}{ \int_{\Omega} \rho |u|^p} +O(\delta)
	$$
	where $\mathcal{D}_k=\{\mathcal{W} \in \X^{s,p}_0(\Omega)\; :\; i(\mathcal{W})\geq k\}$.
	
	We use now the set $\mathcal{C}_{k,\delta}$, which is admissible in the variational characterization of the $k$th--eigenvalue of \eqref{ecux1}, in order to find a bound for it as follows,
	\begin{align} \label{z1.so}
		\lam_{k,p}^\ve  
			\leq  
		\sup_{u \in \mathcal{C}_{k,\delta}} \frac{ [u]_{W^{s,p}(\R^n)}^p}{\int_{\Omega}  \rho_\ve |u|^p} 
			= 
		\sup_{u \in \mathcal{C}_k} \frac{ [u]_{W^{s,p}(\R^n)}^p}{\int_\Omega \rho |u|^p }  \; \frac{\int_\Omega \rho |u|^p}{\int_{\Omega} \rho_\ve |u|^p}.
	\end{align}
	
	To bound $\lam_{k,p}^\ve$ we look for bounds of the two quotients in \eqref{z1.so}. For every function $u\in \mathcal{C}_{k,\delta}$ we have that
	\begin{align} \label{z2.so}
		\frac{  [u]_{W^{s,p}(\R^n)}^p}{ \int_{\Omega} \rho |u|^p} 	\leq 
		\sup_{v \in \mathcal{C}_k} \frac{ [v]_{W^{s,p}(\R^n)}^p}{ \int_{\Omega} \rho |v|^p} 
			= 
		\lam_{k,p} + O(\delta).
	\end{align}
	
	Since $u\in \mathcal{C}_{k,\delta}\subset \X_0^{s,p}(\Omega)$,  by Lemma \ref{lema.sin.orden} we obtain that
	\begin{align} \label{z3.so}
	  \frac{ \int_{\Omega} \rho |u|^p}{\int_{\Omega} \rho_\ve |u|^p} \leq 1+  O(\ve).
	\end{align}
	Then, combining \eqref{z1.so}, \eqref{z2.so} and \eqref{z3.so}  we find that
	$\lam_{k,p}^\ve \leq (\lam_{k,p}+O(\delta))\big( 1+ O(\ve))$, from where it follows that
	\begin{align*}
		\lam_{k,p}^\ve - \lam_{k,p} \leq O(\ve,\delta).
	\end{align*}
	In a similar way, interchanging the roles of $\lam_{k,p}$ and $\lam_{k,p}^\ve$, we obtain that $\lam_{k,p} - \lam_{k,p}^\ve \leq  O(\ve,\delta)$. 
	Gathering both inequalities and letting $\delta\to 0$ and $\ve \to 0$ it is obtained the desired result.	
\end{proof}

\begin{proof}[Proof of Theorem \ref{teo0.n}]
	The proof of the Neumann case it follows with an analogous argument to that of Theorem \ref{teo0} by considering the Rayleigh quotients related to $\Lam_{k,p}$ and $\Lam_{k,p}^\ve$ and by applying Lemma \ref{lema.sin.orden}.
\end{proof}

When periodicity assumptions are made on the weight functions, beside the convergence of the eigenvalues, estimates on the rates of the convergence are obtained. The proofs of Theorems \ref{teo1} and \ref{teo1.n} follow the ideas introduced by Ole{\u\i}nik et al. in \cite{OSY92}, where the problem of obtaining rates on the eigenvalues is reduced to the study of the convergence rates of oscillating integrals. First we prove the Dirichlet case. Later, since the Neumann case involves estimates on the boundary of the domain, it is necessary to assume some additional hypothesis, nevertheless the main idea in the proof still being the same.

The following inequality will be useful to prove our next lemma. We refer to \cite{lind} for the proof.
\begin{lema} \label{lema.p}
For $p>1$ and $x,y\in \R^n$, $x\neq y$,
$$
	|x|^p-|y|^p \leq p|x|^{p-2} x \cdot (x-y).
$$
\end{lema}

\begin{lema} \label{lema.clave}
	Let $\Omega\subset \R^n$ be a bounded domain and denote by $Q$ the unit cube in $\R^n$. Let $g\in L^\infty(\R^n)$ be a $Q$-periodic function such that $\bar g=0$. Then the inequality
	$$
		\left| \int_{\Omega} g(\tfrac{x}{\ve})|v|^p \right| \le C  \ve^s  [v]_{W^{s,2}(\Omega)} \|v\|_{L^p(\Omega)}^{p-1}
	$$
	holds for every $v\in \X^{s,p}_0(\Omega)$ with $s\in(0,1)$. The constant $c$ depends only on $\Omega$, $n$, $p$ and the bounds of $g$.
\end{lema}

\begin{proof}
	Denote by $I^\ve$ the set of all $z\in \Z^n$ such that $Q_{z,\ve}\cap \Omega \neq \emptyset$, $Q_{z,\ve}:=\ve(z+Q)$. Given $v\in \X^{s,p}_0(\Omega)$ we consider the function $\bar v_\ve$ given by the formula
	$$
		\bar{v}_\ve (x)=\frac{1}{\ve^n}\int_{Q_{z,\ve}} v(y)\,dy
	$$
	for $x\in Q_{z,\ve}$. We denote by $\Omega_1 = \cup_{z\in I^\ve} Q_{z,\ve} \supset \Omega$. Thus, we can write
	\begin{align*} 
		\int_{\Omega} g_\ve |v|^p  &=  \int_{\Omega_1} g_\ve (|v|^p-|\bar{v}_\ve|^p) + \int_{\Omega_1} g_\ve |\bar{v}_\ve|^p,
	\end{align*}
	and  we can bound the previous expression as follows
	\begin{equation} \label{keq0.x}
		\int_{\Omega} g_\ve |v|^p 
		\leq
		g_+\int_{\Omega_1}  ||v|^p-|\bar{v}_\ve|^p| + \left|\int_{\Omega_1} g_\ve |\bar{v}_\ve|^p\right|.		
	\end{equation}
	
	The first integral can be split as
	\begin{align} \label{dos}
		\int_{\Omega_1}  ||v|^p-|\bar{v}_\ve|^p|&= \int_{I_1}  |v|^p-|\bar{v}_\ve|^p + \int_{I_2}  |\bar{v}_\ve|^p-|v|^p
	\end{align}
	
	where $I_1=\{x\in\Omega : |v|^p - |\bar v_\ve|^p \geq 0\}$ and $I_2=\{x\in\Omega : |v|^p - |\bar v_\ve|^p < 0\}$. Then, by using Lemma \ref{lema.p} we can bound \eqref{dos} as
	\begin{equation} \label{tres}
		 p\int_{\Omega_1} |v|^{p-1} |v- \bar v_\ve|+p\int_{\Omega_1} |\bar v_\ve|^{p-1}|v-\bar v_\ve|.	
	\end{equation}

	First, observe that by using Lemma \ref{poincarelema} we have
	\begin{align} \label{xxxx1}
	\begin{split}
		\int_{\Omega_1}  |v-\bar{v}_\ve|^p
		&=
			\sum_{z\in I^\ve} \int_{Q_{z,\ve}} |v-\bar{v}_\ve|^p dx \\
		&\leq 
			c^p\ve^{sp} \sum_{z\in I^{\ve}}   [v
		]^p_{W^{s,p}(Q_{z,\ve})}   \\
			& \le 
		c^p\ve^{sp}   [v]^p_{W^{s,p}(\Omega)} .
	\end{split}	
	\end{align}
	
	Secondly,  we can have that
	\begin{align} \label{xxxx2}
			    \int_{\Omega_1}  |v|^{(p-1)p'} = \int_{\Omega_1}  |v|^{p},
	\end{align}
	and since $|\bar v_\ve|\leq \int_{Q} |v| \leq C \Big(\int_{\Omega_1} |v|^p \Big)^\frac{1}{p}$ we get
	\begin{align} \label{xxxx2.2}
	\begin{split}
			    \int_{\Omega_1}  |\bar v_\ve|^{(p-1)p'} \leq  C \int_{\Omega_1} |v|^p
	\end{split}	
	\end{align}	
	where we denote $\frac{1}{p}+\frac{1}{p'}=1$.
	
	From \eqref{xxxx1}, \eqref{xxxx2} and \eqref{xxxx2.2} we bound \eqref{tres} as
	\begin{align} \label{xxxx3}
	\begin{split}
		&
			\left(\int_{\Omega_1}  |v-\bar{v}_\ve|^p\right)^\frac{1}{p} \left( \Big(\int_{\Omega_1}  |v|^{(p-1)p'}\Big)^{p'} + \Big(\int_{\Omega_1}  |\bar v_\ve|^{(p-1)p'}\Big)^{p'}  \right)\\
		&\leq 
			c\ve^{s}   [v]_{W^{s,p}(\Omega)}  \|v\|_{L^p(\Omega)}^{p-1} \\
	\end{split}	
	\end{align}	
	Finally, since $\bar g = 0$ and since $g$ is $Q-$periodic, we get
	\begin{equation}\label{ultima}
		\int_{\Omega_1} g_\ve |\bar v_\ve|^p 
			= 
		\sum_{z\in I^\ve} |\bar v_\ve|^p\mid_{Q_{z,\ve}} \int_{Q_{z,\ve}} g_\ve 
			= 
		0.
	\end{equation}
	Now, combining \eqref{xxxx3} and \eqref{ultima} we can bound \eqref{keq0.x} by
	$$
		\Big|\int_{\Omega} g_\ve |v|^p\Big| \le C  \ve^s  [v]_{W^{s,p}(\Omega)} \|v\|_{L^p(\Omega)}^{p-1},
	$$
	and the proof finishes. 
\end{proof}

Now we are ready to prove the main result.

\begin{proof}[Proof of Theorem \ref{teo1}]
	Let $\delta>0$ and let $\mathcal{C}_{k,\delta}\subset \X^{s,p}_0(\Omega)$ be a set of dimension greater or equal then $k$ such that
	$$
		\lam_{k,p} = 
			\inf_{C\in \mathcal{D}_k} \sup_{u \in C} \frac{  [u]_{W^{s,p}(\R^n)}^p}{\bar{\rho}\int_{\Omega} |u|^p} 
		= 
			\sup_{u \in \mathcal{C}_{k,\delta}} \frac{ [u]_{W^{s,p}(\R^n)}^p}{\bar{\rho}\int_{\Omega} |u|^p} + O(\delta)
	$$
	where $\mathcal{D}_k=\{\mathcal{W} \in \X^{s,p}_0(\Omega)\; :\; i(\mathcal{W})\geq k\}$.
	
	We use now the set $\mathcal{C}_{k,\delta}$, which is admissible in the variational characterization of the $k$th--eigenvalue of \eqref{ecux1}, in order to find a bound for it as follows,
	\begin{align} \label{z1}
		\lam_k^\ve  
			\leq  
		\sup_{u \in \mathcal{C}_{k,\delta}} \frac{  [u]_{W^{s,p}(\R^n)}^p}{\int_{\Omega}  \rho_\ve |u|^p} 
		= 
			\sup_{u \in \mathcal{C}_{k,\delta}} \frac{ [u]_{W^{s,p}(\R^n)}^p}{\bar{\rho}\int_\Omega |u|^p }  \; \frac{\bar{\rho}\int_\Omega |u|^p}{\int_{\Omega} \rho_\ve |u|^p}.
	\end{align}
	
	To bound $\lam_{k,p}^\ve$ we look for bounds of the two quotients in \eqref{z1}. For every function $u\in \mathcal{C}_{k,\delta}$ we have that
	\begin{align} \label{z2}
		\frac{  [u]_{W^{s,p}(\R^n)}^p}{\bar{\rho}\int_{\Omega} |u|^p} \leq \sup_{v \in \mathcal{C}_{k,\delta}} \frac{ [v]_{W^{s,p}(\R^n)}^p}{\bar{\rho}\int_{\Omega} |v|^p} = \lam_{k,p} + O(\delta).
	\end{align}
	
	Since $u\in \mathcal{C}_{k,\delta}\subset \X_0^{s,p}(\Omega)$,  by Lemma \ref{lema.clave} we obtain that
	\begin{align} \label{z3}
	\begin{split}
	  \frac{\bar{\rho}\int_{\Omega} |u|^p}{\int_{\Omega} \rho_\ve |u|^p} 
		  &\leq 
		1+    c \ve^s  \frac{   [u]_{W^{s,p}(\R^n)}\|u\|^{p-1}_{L^p(\Omega)} }{\int_{\Omega} \rho_\ve |u|^p}\\
	  &\leq 
		  1+      c\ve^s \frac{\rho_+}{\rho_-} \frac{[u]_{W^{s,p}(\R^n)} \|u\|^{p-1}_{L^p(\Omega)} }{\bar{\rho} \int_{\Omega}   |u|^p  }\\
	  &\leq 
		  1+      c\ve^s \frac{\rho_+}{\rho_-} \frac{[u]_{W^{s,p}(\R^n)}  }{\bar{\rho} \big(\int_{\Omega}   |u|^p\big)^\frac{1}{p}  }\\		  
	  &\leq 1+  C \ve^s (\lam_{k,p}+O(\delta))^\frac{1}{p}.
	\end{split}
	\end{align}

	Then combining \eqref{z1}, \eqref{z2} and \eqref{z3}  we find that
	$$
		\lam_{k,p}^\ve \leq (\lam_{k,p}+O(\delta))\big( 1+ C  \ve^s \lam_{k,p}^\frac{1}{p} \big).
	$$
	Letting $\delta\to 0$ we get
	\begin{align} \label{z7}
	\lam_{k,p}^\ve - \lam_{k,p} \leq C \ve^s (\lam_{k,p})^{1+\frac{1}{p}}.
	\end{align}
	
	In a similar way, interchanging the roles of $\lam_{k,p}$ and $\lam_{k,p}^\ve$, we obtain that
	\begin{align} \label{z8}
	\lam_{k,p} - \lam_{k,p}^\ve \leq C \ve^s (\lam_{k,p}^\ve)^{1+\frac{1}{p}}.
	\end{align}
	
	Hence, from \eqref{z7} and \eqref{z8}, we arrive at
	\begin{equation*}
	|\lam_{k,p}^\ve - \lam_{k,p}|\le C\ve^s\max\{\lam_{k,p}, \lam_{k,p}^\ve\}^{1+\frac{1}{p}},
	\end{equation*}
	and by using the bounds given \eqref{lapf.2} and  \eqref{lapf.1} the result follows.
\end{proof}

The following Lemma is necessary to deal with the convergence rates of functions in $W^{s,p}(\Omega)$.

\begin{lema} \label{lema.n}
	Let $\Omega\subset \R^n$ be a bounded domain with $C^1$ boundary and, for $\delta > 0$, let $G_\delta$ be a tubular
	neighborhood of $\partial \Omega$, i.e. $G_\delta = \{x\in\Omega \,:\,  dist(x, \partial\Omega) < \delta\}$. Then there exists $\delta_0 > 0$ such that for every $\delta\in(0, \delta_0)$ and every $v \in W^{s,p}(\Omega)$ we have
	$$
		\|u\|_{L^p(G_\delta)}\leq c\delta^\frac{1}{p} \|u\|_{W^{s,p}(\Omega)}.
	$$
	whenever $\frac{1}{p}<s<  1$.
\end{lema}
\begin{proof}
	Let $G_\delta=\{x\in \Omega \: \, dist(x,\partial \Omega)<\delta\}$, it follows that $\partial G_\delta$ are uniformly smooth surfaces. By the trace's inequality stated in Proposition \ref{traza} we have
	\begin{align*}
		\|u\|^p_{L^p(\partial G_\delta)} &\leq \|u\|_{W^{s-\frac{1}{p},p}(\partial G_\delta)}\\
			&\leq
		 c\|u\|_{W^{s,p}(G_\delta)}\\
		&\leq 
		c \|u\|_{W^{s,p}(\Omega)},	 \qquad \delta\in(0,\delta_0)
	\end{align*}
	provided that $\frac{1}{p}<s< 1$, where $c$ is a constant independent on $\delta$ and $u$. Integrating this inequality with respect to $\delta$ we get
	$$
		\|u\|^p_{L^p(G_\delta)}=\int_0^\delta \Big( \int_{\partial G_\tau} |u|^p \, dS\Big)\, d\tau \leq c\delta \|u\|^p_{W^{s,p}(\Omega)}
	$$
	and the result is proved.
\end{proof}

The proof of the next Lemma follows with a slight modification to that of Lemma \ref{lema.clave}, and it is essential in order to handle the convergence rates of eigenvalues of the Neumann problem \eqref{ecux2.n}.

\begin{lema} \label{lema.clave.n}
	Let $\Omega\subset \R^n$ be a bounded domain with $C^1$ boundary and denote by $Q$ the unit cube in $\R^n$. Let $g\in L^\infty(\R^n)$ be a $Q$-periodic function such that $\bar g=0$. Then the inequality
	$$
		\left| \int_{\Omega} g(\tfrac{x}{\ve})|v|^p \right| \leq c \ve^s \|v\|_{W^{s,p}(\Omega)}
	$$
	holds for every $v\in W^{s,p}(\Omega)$ with $\frac{1}{p}<s<1$. The constant $c$ depends only on $\Omega$, $p$, $n$ and the bounds of $g$.
\end{lema}

\begin{proof}
	The proof is quite similar to that of Lemma \ref{lema.clave}, however there are some details to have into account. 	Denote by $I^\ve$ the set of all $z\in \Z^n$ such that $Q_{z,\ve}\cap \Omega \neq \emptyset$, $Q_{z,\ve}:=\ve(z+Q)$. Given $v\in W^{s,p}(\Omega)$ we consider the function $\bar v_\ve$ given by the formula
	$$
		\bar{v}_\ve (x)=\frac{1}{\ve^n}\int_{Q_{z,\ve}} v(y)\,dy
	$$
	for $x\in Q_{z,\ve}$. We denote by $\Omega_1 = \cup_{z\in I^\ve} Q_{z,\ve} \supset \Omega$. Thus, we can write
	\begin{align} \label{keq0}
	\begin{split}
		\int_{\Omega} g_\ve |v|^p  &= \int_G g_\ve |v|^p + \int_{\Omega_1} g_\ve (|v|^p-|\bar{v}_\ve|^p) + \int_{\Omega_1} g_\ve |\bar{v}_\ve|^p.
	\end{split}	
	\end{align}	
	where $G=\Omega\setminus \bar\Omega_1$.

	As in the Dirichlet Lemma we have that
	\begin{align} \label{xxxx1.n}
	\begin{split}
		\int_{\Omega_1} g_\ve (|v|^p-|\bar{v}_\ve|^p) + \int_{\Omega_1} g_\ve |\bar{v}_\ve|^p
			  \le 
		c\ve^{s}   [v]_{W^{s,p}(\Omega)} \|u\|^{p-1}_{L^p(\Omega)}.
	\end{split}	
	\end{align}		
	
 	The set $G$ is a $\delta-$neighborhood of $\partial \Omega$ with $\delta=c\ve$ for some constant $c$, and therefore, according to Lemma \ref{lema.n} we have
	\begin{equation} \label{yyy.n}
		\int_G g_\ve |v|^p \leq c\ve \|v\|^p_{W^{s,p}(\Omega)}.
	\end{equation}

	Since $\ve$ and $s$ are lower than 1, gathering \eqref{keq0}, \eqref{xxxx1.n} and  \eqref{yyy.n} we obtain that
	\begin{align*}
			\Big|\int_{\Omega} g_\ve |v|^p\Big| &\le C \ve^s  [v]_{W^{s,p}(\Omega)} \|u\|^{p-1}_{L^p(\Omega)}+ C\ve \|v\|_{W^{s,p}(\Omega)}^p	\\
			&\leq C\ve^s \|v\|_{W^{s,p}(\Omega)}^p
	\end{align*}
	and the proof finishes. 
\end{proof}

 Having been proved Lemma \ref{lema.clave.n}, the proof of Theorem \ref{teo1.n} is analogous to that of Theorem \ref{teo1} by using Lemma \ref{lema.clave.n} instead of Lemma \ref{lema.clave} together with the bound  given in \eqref{lapn}.

\bibliographystyle{amsplain}

\begin{thebibliography}{10}

\bibitem{BLP78} A. Bensoussan, J.-L. Lions, and G. Papanicolaou, \emph{Asymptotic
    analysis for
  periodic structures}, AMS Chelsea Publishing, Providence, RI, 2011, Corrected
  reprint of the 1978 original [MR0503330]. \MR{2839402}

\bibitem{BG59}
R.~M. Blumenthal and R.~K. Getoor, \emph{The asymptotic distribution of the
  eigenvalues for a class of {M}arkov operators}, Pacific J. Math. \textbf{9}
  (1959), 399--408. \MR{0107298 (21 \#6023)}
  
\bibitem{BM76}
L. Boccardo and P. Marcellini, \emph{Sulla convergenza delle soluzioni di
  disequazioni variazionali}, Ann. Mat. Pura Appl. (4) \textbf{110} (1976),
  137--159. \MR{0425344 (54 \#13300)}
  

\bibitem{Ch-dP}
T. Champion and L. De~Pascale, \emph{Asymptotic behaviour of nonlinear
  eigenvalue problems involving {$p$}-{L}aplacian-type operators}, Proc. Roy.
  Soc. Edinburgh Sect. A \textbf{137} (2007), no.~6, 1179--1195. \MR{2376876
  (2009b:35315)}

\bibitem{CoHi}
Courant, R., Hilbert, D., \emph{Methoden der mathematischen Physik, vol. 1 (Vol. 2, p. 1937)}, (1931).   Berlin.


 

\bibitem{CS05}
Zhen-Qing Chen and Renming Song, \emph{Two-sided eigenvalue estimates for
  subordinate processes in domains}, J. Funct. Anal. \textbf{226} (2005),
  no.~1, 90--113. \MR{2158176 (2006d:60116)}

\bibitem{DPS15}
Del Pezzo, L. M., Salort, A. M., \emph{The first non-zero Neumann p-fractional eigenvalue.}, (2015). Nonlinear Analysis: Theory, Methods \& Applications, 118, 130-143.


\bibitem{DD12}
Demengel, F., Demengel, G., \emph{Functional spaces for the theory of elliptic partial differential equations.}, (2012). Springer.





\bibitem{DNPV}
Eleonora Di~Nezza, Giampiero Palatucci, and Enrico Valdinoci,
  \emph{Hitchhiker's guide to the fractional {S}obolev spaces}, Bull. Sci.
  Math. \textbf{136} (2012), no.~5, 521--573. \MR{2944369}

\bibitem{FBPS13}
J. Fern{\'a}ndez~Bonder, J.P. Pinasco, and A.M. Salort,
  \emph{Convergence rate for some quasilinear eigenvalues homogenization problems}, (2015) Journal of Mathematical Analysis and Applications, 423(2), 1427-1447.

\bibitem{IS}
Iannizzotto, A., Squassina, M. (2014). \emph{Weyl-type laws for fractional p-eigenvalue problems}. Asymptot. Anal, 88(4), 233-245.  

\bibitem{KLZ13}
C. Kenig, F. Lin, and Z. Shen, \emph{Estimates of eigenvalues
  and eigenfunctions in periodic homogenization}, J. Eur. Math. Soc. (JEMS)
  \textbf{15} (2013), no.~5, 1901--1925. \MR{3082248}
  
\bibitem{Ke79}
S. Kesavan, \emph{Homogenization of elliptic eigenvalue problems.
  {II}}, Appl. Math. Optim. \textbf{5} (1979), no.~3, 197--216. \MR{546068
  (80i:65110)}
  

\bibitem{KLZ12}
C. Kenig, F. Lin, and Z. Shen, \emph{Convergence rates in
  {$L^2$} for elliptic homogenization problems}, Arch. Ration. Mech. Anal.
  \textbf{203} (2012), no.~3, 1009--1036. \MR{2928140}
  

\bibitem{lind}
P. Lindqvist, \emph{On the equation {${\rm div}\,(\vert \nabla u\vert ^{p-2}\nabla u)+\lambda\vert u\vert ^{p-2}u=0$}}, Proc. Amer. Math. Soc.
  \textbf{109} (1990), no.~1, 157--164. \MR{1007505 }

\bibitem{MMP}  
D. Motreanu, V.V. Motreanu, N.S. Papageorgiou, \emph{Topological and variational methods with applications
to nonlinear boundary value problems}, Springer, New York, xi+459 pp. (2014).
    
 
\bibitem{OSY92}
O.~A. Ole{\u\i}nik, A.~S. Shamaev, and G.~A. Yosifian, \emph{Mathematical
  problems in elasticity and homogenization}, Studies in Mathematics and its
  Applications, vol.~26, North-Holland Publishing Co., Amsterdam, 1992.
  \MR{1195131 (93k:35025)}  

\bibitem{SP70}
E. S{\'a}nchez-Palencia, \emph{\'{E}quations aux d\'eriv\'ees partielles
  dans un type de milieux h\'et\'erog\`enes}, C. R. Acad. Sci. Paris S\'er. A-B 
 
\bibitem{sch}
Schneider, C. (2008). \emph{Trace operators in Besov and Triebel-Lizorkin spaces}. Univ. Leipzig, Fak. für Mathematik u. Informatik. 
 
 
\bibitem{SV13}
Servadei, R.,  Valdinoci, E. \emph{Variational methods for non-local operators of elliptic type} (2013). Discrete Contin. Dyn. Syst, 33(5), 2105-2137.  
  


 
\end{thebibliography}

\end{document}